\documentclass[10pt]{amsart}
\usepackage{latexsym, amsmath,amssymb}
\usepackage{enumitem}

\renewcommand{\epsilon}{\varepsilon}
\newcommand{\newsection}[1]
{\subsection{#1}\setcounter{theorem}{0} \setcounter{equation}{0}
\par\noindent}

\newtheorem{theorem}{Theorem}

\newtheorem{lemma}[theorem]{Lemma}

\newtheorem{proposition}[theorem]{Proposition}

\theoremstyle{definition}
\newtheorem{deff}[theorem]{Definition}

\theoremstyle{remark}
\newtheorem{remark}[theorem]{Remark}

\newcommand{\supp}{\operatorname{supp}}

\renewcommand{\epsilon}{\varepsilon}

\newcommand{\R}{{\mathbb R}}
\newcommand{\Z}{{\mathbb Z}}

\newcommand{\tgamma}{\tilde \gamma}
\newcommand{\bfk}{\mathbf k}

\newcommand{\lap}[1]{\sqrt{-\Delta_{#1}}}

\begin{document}


\title[Integrals of Eigenfunctions over Curves]
{Integrals of eigenfunctions over curves in surfaces of nonpositive curvature}

	\author{Emmett L. Wyman}
	\address{Department of Mathematics,  Johns Hopkins University,
	Baltimore, MD 21218}

\begin{abstract}  Let $(M,g)$ be a compact, 2-dimensional Riemannian manifold with nonpositive sectional curvature. Let $\Delta_g$ be the Laplace-Beltrami operator corresponding to the metric $g$ on $M$, and let $e_\lambda$ be $L^2$-normalized eigenfunctions of $\Delta_g$ with eigenvalue $\lambda$, i.e.
\[
-\Delta_g e_\lambda = \lambda^2 e_\lambda.
\]
We prove
\[
\left| \int_\R b(t) e_\lambda (\gamma(t)) \, dt \right| = o(1) \quad \text{ as } \lambda \to \infty
\]
where $b$ is a smooth, compactly supported function on $\R$ and $\gamma$ is a curve parametrized by arc-length whose geodesic curvature $\kappa(\gamma(t))$ avoids two critical curvatures $\mathbf k(\gamma'^\perp(t))$ and $\mathbf k(-\gamma'^{\perp}(t))$ for each $t \in \supp b$. $\bfk(v)$ denotes the curvature of a circle with center taken to infinity along the geodesic ray in direction $-v$.

Chen and Sogge prove in ~\cite{ChenSogge} the same decay for geodesics in $M$ with strictly negative curvature. After performing a standard reduction, they lift the relevant quantity to the universal cover and then use the Hadamard parametrix to reduce the problem to bounding a sum of oscillatory integrals with a geometric phase functions. They use the Gauss-Bonnet theorem to obtain bounds on the Hessian of these phase functions and conclude their argument with stationary phase. Our argument follows theirs, except we prove and use properties of the curvature of geodesic circles to obtain bounds on the Hessian of the phase functions.
\end{abstract}

\maketitle

\newsection{Statement of results}


Let $(M,g)$ be a 2-dimensional compact Riemannian manifold. We denote by $e_\lambda$ an $L^2$-normalized eigenfunction of the Laplace-Beltrami operator $\Delta_g$ on $M$, i.e. $-\Delta_g e_\lambda = \lambda^2 e_\lambda$ and $\| e_\lambda \|_{L^2(M)} = 1$. We are interested in restrictions of eigenfunctions to curves in $M$, in particular with the integral
\begin{equation} \label{intro1}
    \int b(t) e_\lambda(\gamma(t)) \, dt
\end{equation}
where $b$ is a smooth, compactly supported function on $\R$ and $\gamma$ is a smooth unit-speed curve in $M$. In the setting that $M$ is a hyperbolic surface and $\gamma$ is a closed geodesic, Good ~\cite{Good} and Hejhal ~\cite{Hejhal} showed that
\[
    \int_\gamma e_\lambda \, dt = O(1).
\]
Later Reznikov ~\cite{Rez} demonstrated the same bound can be achieved if $\gamma$ is allowed to be a circle in $M$. For $M$ of arbitrary dimension, Zelditch ~\cite{ZelK} shows, among other things, period integrals over submanifolds of codimension $k$ are $O(\lambda^{\frac{k-1}{2}})$, implying the $O(1)$ bound above.

In the setting where $M$ has negative sectional curvature, Chen and Sogge ~\cite{ChenSogge} obtained decay
\begin{equation} \label{little o}
    \int b(s) e_\lambda(\gamma(s)) \, ds = o(1)
\end{equation}
where $\gamma$ is a geodesic in $M$. Moreover, they showed that decay cannot be guaranteed if $M$ is replace with a sphere or a torus, demonstrating the necessity of negative sectional curvature. In the case of the sphere, the bound is saturated by the zonal functions along the equator. In the case of the torus, for any closed geodesic $\gamma$ there exists a sequence of eigenfunctions which are uniformly constant on $\gamma$. Sogge, Xi, and Zhang ~\cite{SXZ} later improved this result by slightly weakening the hypotheses on the curvature of $M$ and obtaining an explicit decay of $O((\log \lambda)^{-1/2})$.

Our main result builds on the work of Chen and Sogge ~\cite{ChenSogge} and shows that their bound \eqref{little o} holds for integrals over $\gamma$ belonging to a wider class of curves.

\noindent \textbf{Notation.} For a $2$-dimensional Riemannian manifold $M$, we let $K(p)$ denote the sectional curvature of $M$ at a point $p \in M$.
Let $\gamma$ be a regular parametrized curve in $M$. We let $\kappa_\gamma(t)$ denote the geodesic curvature of $\gamma$ at $t$,
\begin{equation*} \label{def geodesic curvature}
\kappa_\gamma(t) = \frac{1}{|\gamma'(t)|} \left| \frac{D}{dt} \frac{\gamma'(t)}{|\gamma'(t)|} \right|,
\end{equation*}
where $D/dt$ denotes the covariant derivative in the variable $t$.
For any point $p \in M$ and $v \in T_pM$, we let $v^\perp$ denote a choice of vector in $T_p M$ such that $|v^\perp| = |v|$ and $\langle v, v^\perp \rangle = 0$. $SM = \{ v \in TM : |v| = 1 \}$ denotes the unit sphere bundle over $M$.

Essential to our result is a particular function $\mathbf k$ on the unit sphere bundle $SM$, defined below.

\begin{deff} \label{def k} Let $(M,g)$ be a $2$-dimensional Riemannian manifold, without boundary, with non-positive sectional curvature. Let $v \in SM$ and $\zeta$ be the geodesic with $\zeta'(0) = v$, and $J$ be a Jacobi field along $\zeta$ satisfying
\begin{equation} \label{J initial condition}
|J(0)| = 1 \quad \text{ and } \quad \langle J(0), \zeta'(0) \rangle = 0.
\end{equation}
We denote by $\mathbf k(v)$ the unique value such that
\begin{equation} \label{J bounded}
|J(r)| = O(1) \quad \text{ for } r \leq 0
\end{equation}
if $J$ satisfies the additional initial condition
\begin{equation} \label{J' initial condition}
\frac{D}{dr} J(0) = \mathbf k(v) J(0).
\end{equation}
\end{deff}

We verify that $\mathbf k$ is well-defined, continuous, and non-negative in Proposition ~\ref{verify def k}. The geometric meaning of $\mathbf k$ is clearer after pulling it back to the universal cover of $M$. By the theorem of Hadamard, we identify the universal cover of $(M,g)$ with $(\R^2,\tilde g)$, where $\tilde g$ is the pullback of $g$ through the covering map. If $v$ and $\zeta$ are as in the definition and $\tilde \zeta$ is a lift of $\zeta$ to $\R^2$, then $\mathbf k(v)$ denotes the limiting curvature of a circle at $\tilde \zeta(0)$ with center at $\tilde \zeta(-R)$ as $R \to \infty$. This fact comes out in the proof of Proposition \ref{verify def k} and Remark \ref{k lift}. Our main result is as follows.

\begin{theorem} \label{main theorem}
Let $(M,g)$ be a compact $2$-dimensional Riemannian manifold without boundary and with nonpositive sectional curvature. Let $b$ be a smooth function on $\R$ with compact support and $\gamma$ be a smooth unit-speed curve satisfying
\begin{equation} \label{main theorem hypotheses}
\kappa_\gamma(t) \neq \mathbf k(\gamma'^\perp (t)) \quad \text{ and } \quad \kappa_\gamma(t) \neq \mathbf k(- \gamma'^\perp(t)) \qquad \text{ for all } t \in \supp b.
\end{equation}
Then,
\begin{equation} \label{main theorem conclusion}
\int b(t) e_\lambda(\gamma(t)) \, dt = o(1)
\end{equation}
as $\lambda \to \infty$.
\end{theorem}

If $M = \R^2/2\pi \Z^2$ is the flat torus, one can check directly from the definition that $\mathbf k \equiv 0$, and so $\gamma$ must have nonvanishing curvature by \eqref{main theorem hypotheses}. In fact, much stronger decay can be obtained on the torus in this situation. We write
\[
    e_\lambda(x) = \sum_{|m| = \lambda} a(m) e^{i x \cdot m}
\]
where $m \in \Z^2$ and
\[
    \sum_{|m| = \lambda} |a(m)|^2 = 1.
\]
Hence by Cauchy-Schwarz
\[
    \left| \int b(t) e_\lambda(\gamma(t)) \, dt \right| \leq \#\{ m \in \Z^2 : |m| = \lambda \}^{1/2} \sup_{|m| = \lambda} \left| \int b(t) e^{i \gamma(t) \cdot m} \, dt \right|.
\]
Since $\gamma$ has nonvanishing curvature, an elementary stationary phase argument tells us the supremum in the line above is $O(\lambda^{-1/2})$. Bounds on the divisor function in the Gaussian integers give us
\[
    \#\{ m \in \Z^2 : |m^2| = \lambda^2 \} = O(\lambda^{\epsilon})
\]
for any fixed $\epsilon > 0$. Hence, we obtain $O(\lambda^{-1/2+\epsilon})$ decay for \eqref{main theorem conclusion} for the torus. This result is essentially sharp as demonstrated by taking $\gamma$ to be a circle and $b \equiv 1$.

Another special case is when $M$ is a compact hyperbolic surface, i.e. $M$ has constant sectional curvature $-1$. Then, $\mathbf k \equiv 1$. The hypotheses \eqref{main theorem hypotheses} then exclude curves that lift to horocycles in the universal cover. As in ~\cite{Helgason}, the characters used in the Fourier transform on the hyperbolic plane are constant on families of horocycles. The author would be interested to know of an example of a compact hyperbolic surface and $\gamma$ with curvature $1$ such that the integral of eigenfunctions over $\gamma$ saturate the $O(1)$ bound, i.e.
\[
    \limsup_{\lambda \to \infty} \left| \int b(t) e_\lambda(\gamma(t)) \, dt \right| > 0.
\]

To prove our main result, we follow Chen and Sogge's strategy exactly as in ~\cite{ChenSogge}. First, we make a reduction using the Cauchy-Schwarz inequality to phrase the bound in \eqref{main theorem conclusion} as a kernel bound. Second, we lift the problem to the universal cover where we will use a lemma from ~\cite{ChenSogge} to write the kernel as a sum of oscillatory integrals. In ~\cite{ChenSogge}, Chen and Sogge use the Gauss-Bonnet theorem to obtain bounds on the derivatives of the phase function and conclude their argument with stationary phase. We obtain bounds on the derivatives of the phase function by exploiting our hypotheses on $\gamma$ and the behavior of the curvature of large circles in the universal cover.

\noindent \textbf{Acknowledgements.}  The author would like to thank his advisor, Christopher Sogge, for providing the initial problem, related materials, feedback, and support\footnote{This work is partially supported by the NSF.}. The author would also like to thank Yakun Xi and Cheng Zhang for their feedback.

\newsection{Standard reduction and lift to the universal cover}

We use Chen and Sogge's argument in ~\cite{ChenSogge} to reduce the bound in \eqref{main theorem conclusion} to two stationary phase arguments, Propositions \ref{local bound} and \ref{global bound}, which we prove using the tools developed in the previous section.

Let $\rho \in C^\infty(\R)$ be a smooth function satisfying $\rho(0) = 1$ and $\supp \hat \rho \subset [-1/2,1/2]$. For any $T > 1$, we define the operator $\rho(T(\lap g - \lambda))$ using the spectral theorem, i.e.
\[
	\rho(T(\lap g - \lambda)) f = \sum_j \rho(T(\lambda_j - \lambda)) E_j f
\]
where $E_j$ is the orthogonal projection of $f$ onto the space spanned by $e_j$. To prove Theorem \ref{main theorem}, it suffices to show
\begin{equation} \label{mainbound}
	\left| \int b(t) \rho(T(\lap g - \lambda)) f(\gamma(t)) \, dt \right| \leq (CT^{-1} + C_T \lambda^{-1/2})^{1/2} \| f \|_{L^2(M)}.
\end{equation}
Where $C$ is a fixed constant and $C_T$ is some constant depending on $T$.
Using
\[
	\rho(T(\lap g - \lambda))f(x) = \int_M \left( \sum_j \rho(T(\lambda_j - \lambda)) e_j(x) \overline{e_j(y)} \right) f(y) \, dV(y),
\]
Cauchy-Schwarz, and orthogonality\footnote{The Cauchy-Schwarz reduction here occurred earlier in ~\cite{ChenSogge} and ~\cite{SXZ}.}, we write the integral in \eqref{mainbound} as
\begin{align*}
	\left| \int_M \int \sum_j \right. & \left. \vphantom{\int_M \sum_j} b(t)  \rho(T(\lambda_j - \lambda)) e_j(\gamma(t)) \overline{e_j(y)} f(y) \, dt \, dV(y) \right| \\
	&\leq \left( \int_M \left| \int \sum_j b(t) \rho(T(\lambda_j - \lambda)) e_j(\gamma(t)) \overline{e_j(y)} \, dt \right|^2 \, dV(y) \right)^{1/2} \| f \|_{L^2(M)}\\
	&= \left| \iint \sum_j b(s,t) \chi(T(\lambda_j - \lambda)) e_j(\gamma(s)) \overline{e_j(\gamma(t))} \, ds \, dt \right|^{1/2} \| f \|_{L^2(M)}
\end{align*}
where $b(s,t) = b(s) b(t)$ and $\chi = |\rho|^2$. Note that $\supp \hat \chi \subset [-1,1]$. Hence \eqref{mainbound} would follow if we could show
\begin{equation} \label{mainbound2}
	\left| \iint \sum_j b(s,t) \chi(T(\lambda_j - \lambda)) e_j(\gamma(s)) \overline{e_j(\gamma(t))} \, ds \, dt \right| \leq CT^{-1} + C_T \lambda^{-1/2}
\end{equation}
By Fourier inversion and a change of variables, we have
\begin{align*}
	\sum_j \chi(T(\lambda_j - \lambda)) e_j(x) \overline{e_j(y)} &= \frac{1}{2\pi} \int \sum_j \hat \chi(\tau) e^{i\tau T(\lambda_j - \lambda)} e_j(x) \overline{e_j(y)} \, d\tau \\
	&= \frac{1}{2\pi T} \int \sum_j \hat \chi(\tau/T) e^{i\tau(\lambda_j - \lambda)} e_j(x) \overline{e_j(y)} \, d\tau\\
	&= \frac{1}{2\pi T} \int  \hat \chi(\tau/T) e^{-i\tau\lambda} e^{i\tau \lap g}(x,y) \, d\tau,
\end{align*}
where the last line follows from writing out the kernel of the half-wave operator $e^{i\tau \lap g}$,
\[
	e^{i\tau \lap g}(x,y) = \sum_j e^{i\tau \lambda_j} e_j(x) \overline{e_j(y)}.
\]
Hence, we write \eqref{mainbound2} as
\begin{equation} \label{mainbound3}
	\left| \iiint b(s,t) \hat \chi(\tau/T) e^{-i \tau \lambda} e^{i \tau \lap g}(\gamma(s), \gamma(t)) \, d\tau \, ds \, dt \right| \leq C + C_T \lambda^{-1/2}.
\end{equation}

At this point, we let $\beta \in C_0^\infty(\R)$ with $\beta(\tau) = 1$ if $|\tau| \leq 3$ and $\beta(\tau) = 0$ if $|\tau| \geq 4$. By scaling the metric, we can assume the injectivity radius of $M$ is $10$ or more, and by a partition of unity, we may restrict the support of $b$ to lie in an interval of length $1$. We write
\begin{align*}
\int \hat \chi(\tau/T) e^{-i\tau \lambda} e^{i\tau \lap g} (x,y) \, d\tau &= \int \beta(\tau) \hat \chi(\tau/T) e^{-i\tau \lambda} e^{i\tau \lap g} (x,y) \, d\tau\\
&\quad + \int (1 - \beta(\tau)) \hat \chi(\tau/T) e^{-i\tau \lambda} e^{i\tau \lap g} (x,y) \, d\tau.
\end{align*}
We claim the contribution of the $\beta$ part to the integral in \eqref{mainbound3} is $O(1)$. As noted in ~\cite{ChenSogge} and ~\cite{SXZ}, by the proof of Lemma 5.1.3 in ~\cite{FIOs} and the assumption that the injectivity radius of $M$ is at least $10$ we can write this term as
\[
\int \beta(\tau) \hat \chi(\tau /T) e^{-i\tau \lambda} e^{i\tau \lap g} (x,y) \, d\tau = \lambda^{1/2} \sum_\pm a_\pm(\lambda; d_g(x,y)) e^{\pm i \lambda d_g(x,y)} + O(1)
\]
where $a_\pm$ satisfies bounds
\begin{equation} \label{r > 1/lambda}
\left| \frac{d^j}{dr^j} a_\pm(\lambda; r) \right| \leq C_j r^{-j - 1/2} \quad \text{ if } r \geq \lambda^{-1},
\end{equation}
and
\begin{equation} \label{r < 1/lambda}
|a_\pm(\lambda; r)| \leq C \lambda^{1/2} \quad \text{ if } 0 \leq r \leq \lambda^{-1}.
\end{equation}
Our claim follows if
\begin{equation} \label{beta contribution}
\lambda^{1/2} \iint b(s,t) a_\pm(\lambda; d_g(\gamma(s),\gamma(t))) e^{\pm i \lambda d_g(\gamma(s), \gamma(t))} \, ds \, dt = O(1).
\end{equation}
After perhaps further restricting the support of $b$, we have by the inverse function theorem a smooth change of variables $(s,r) \mapsto (s,t(s,r))$ where
\[
r = \begin{cases} d_g(\gamma(s), \gamma(t)) & \text{if } s \geq t \\
-d_g(\gamma(s), \gamma(t)) & \text{if } s \leq t.
\end{cases}
\]
We then rewrite the integral in \eqref{beta contribution} as
\begin{equation} \label{post coc} \nonumber
\lambda^{1/2} \left( \iint_{|r| \leq \lambda^{-1}} + \iint_{|r| > \lambda^{-1}} \right) \tilde b(s,r) a_\pm (\lambda ; |r|) e^{\pm i \lambda |r|} \, ds \, dr
\end{equation}
where we use $\tilde b(s,r) \, ds \, dr$ to denote $b(s,t) \, ds \, dt$. The $|r| \leq \lambda^{-1}$ part is trivially $O(1)$ by \eqref{r < 1/lambda}. The $|r| > \lambda^{-1}$ part is also $O(1)$ after integrating by parts once in $r$ and applying \eqref{r > 1/lambda}. Hence we have \eqref{beta contribution}, and what is left is to show
\begin{equation} \label{mainbound4}
	\left| \iiint b(s,t) (1 - \beta(\tau)) \hat \chi(\tau/T) e^{-i \tau \lambda} e^{i \tau \lap g}(\gamma(s), \gamma(t)) \, d\tau \, ds \, dt \right| \leq C + C_T \lambda^{-1/2}.
\end{equation}

We will need to lift the computation to the universal cover. Before we do this, we want to rephrase \eqref{mainbound4} using $\cos(\tau \lap g)$ rather than $e^{i\tau \lap g}$. This will allow us to make use of Huygen's principle after we lift to ensure the kernel we obtain is supported on a neighborhood of the diagonal. Using Euler's formula, we write
\begin{align*}
	\int (1 - \beta(\tau)) &\hat \chi(\tau/T) e^{-i\tau \lambda} e^{i \tau \lap g}(x,y) \, d\tau \\
	&= 2 \int (1 - \beta(\tau)) \hat\chi(\tau/T)  e^{-i\tau \lambda} \cos(\tau \lap g)(x,y) \, d\tau\\
	&\quad - \int (1 - \beta(\tau)) \hat \chi(\tau/T) e^{-i \tau \lambda} e^{-i \tau \lap g}(x,y) \, d\tau.
\end{align*}
Writing $\hat \Phi_T(\tau) = (1 - \beta(\tau)) \hat \chi(\tau/T)$, the latter term becomes
\[
\sum_j \int (1 - \beta(\tau)) \hat \chi(\tau/T) e^{-i\tau(\lambda_j + \lambda)} e_j(x) \overline{e_j(y)} \, d\tau = \frac{1}{2\pi} \sum_j \Phi_T(\lambda_j + \lambda) e_j(x) \overline{e_j(y)}.
\]
The contribution from this term to the integral in \eqref{mainbound4} is rapidly decaying in $\lambda$, uniformly in $T$. Hence, it suffices to show
\begin{align} \label{cosinebound}
	\left| \iiint b(s,t) (1 - \beta(\tau)) \hat \chi(\tau/T) e^{-i \tau \lambda} \cos(\tau \lap g)(\gamma(s), \gamma(t)) d\tau \, ds \, dt \right| \leq C + C_T \lambda^{-1/2}
\end{align}

We are now ready to lift to the universal cover. We identify the universal cover of $M$ with $\R^2$ equipped with the pullback metric $\tilde g$. Let $\Gamma$ be the group of deck transformations. Let $\tilde f \in C^\infty_0(\R^2)$ and
\[
	f(p) = \sum_{\alpha \in \Gamma} \tilde f(\alpha (\tilde p)),
\]
where $\tilde p$ is a lift of $p$ through the covering map. Now let $\tilde u(\tilde p, t)$ be the solution to the wave equation $\square_{\tilde g} \tilde u = 0$ with initial data $u(\tilde p,0) = \tilde f(\tilde p)$. Let $u(p,t) = \sum_{\alpha \in \Gamma} \tilde u(\alpha(\tilde p), t)$. Observe that $u$ satisfies the wave equation $\square_g u = 0$ with initial data $u(p,0) = f(p)$. Hence, we conclude that
\[
	\cos(\tau \lap{g}) f(p) = \sum_{\alpha \in \Gamma} \cos(\tau \lap{\tilde g}) \tilde f(\alpha(\tilde p)),
\]
and so we have
\[
	\cos(\tau \lap{g})(x,y) = \sum_{\alpha \in \Gamma} \cos(\tau \lap{\tilde g})(\alpha(\tilde x), \tilde y),
\]
where $\tilde x$ and $\tilde y$ are lifts of $x$ and $y$ through the covering map, respectively. Hence, we write \eqref{cosinebound} as
\begin{equation} \label{unicoverbound}
\left| \sum_{\alpha \in \Gamma} \iint b(s,t) K_{T,\lambda}(\tilde \gamma_\alpha(s), \tilde \gamma(t)) \, ds \, dt \right| \leq C + C_T \lambda^{-1/2}
\end{equation}
where
\begin{equation} \label{kernel definition} \nonumber
K_{T,\lambda}(x,y) = \int (1 - \beta(\tau)) \hat \chi(\tau/T) e^{-i\tau \lambda} \cos(\tau \lap{ \tilde g})(x,y) \, d\tau,
\end{equation}
where $x$ and $y$ belong to the universal cover.
Here $\tilde \gamma$ is a lift of $\gamma$ to the universal cover, and $\tilde \gamma_\alpha = \alpha \circ \tilde \gamma$. Now $\cos(\tau \lap{\tilde g})(x,y)$ is supported on $d_g(x,y) \leq |\tau|$ by Huygen's principle, and since $\hat \chi(\tau/T)$ is suppoted on $[-T,T]$, we have that $K_{T,\lambda}$ is supported on $d_{\tilde g}(x,y) \leq T$. Hence, the sum in \eqref{unicoverbound} is finite. In fact, as noted in ~\cite{ChenSogge} and ~\cite{SXZ}, the sum has $O(e^{CT})$ terms by volume comparison.

To proceed, we will need bounds on $K_{T,\lambda}$. We will make use of Lemma 2.4 from ~\cite{ChenSogge}, stated below.

\begin{lemma}[Chen and Sogge] We write
\[
K_{T,\lambda}(\tilde x, \tilde y) = \lambda^{1/2} w( \tilde x, \tilde y) \sum_\pm a_\pm(T,\lambda; d_{\tilde g}(\tilde x, \tilde y)) e^{\pm i \lambda d_{\tilde g}(\tilde x, \tilde y)} + R_{T,\lambda}( \tilde x, \tilde y)
\]
where $w$ is a smooth bounded function on $\R^2 \times \R^2$ and where for each $j = 0,1,2,\ldots$ there is a constant $C_j$ independent of $T,\lambda \geq 1$ so that
\begin{equation} \label{a bounds}
\left| \frac{d^j}{dr^j} a_\pm(T,\lambda; r) \right| \leq C_j r^{-1/2 - j}, \quad \text{ for } r \geq 1,
\end{equation}
and for a constant $C_T$ independent of $\gamma$ and $\lambda$ such that
\[
|R_{T,\lambda}(\tilde x, \tilde y)| \leq C_T \lambda^{-1}.
\]
\end{lemma}
The contribution of the $R_{T,\lambda}$ to the sum in \eqref{unicoverbound} is bounded by $C_T \lambda^{-1}$, better than required. Moreover since $\beta(\tau) = 1$ if $|\tau| \leq 3$, we have $(1 - \beta(\tau)) \cos(\tau \lap{\tilde g})(\tilde x, \tilde y)$ is smooth if $d_{\tilde g}( \tilde x, \tilde y) \leq 1$. Hence for $d_{\tilde g}( \tilde x, \tilde y) \leq 1$,
\[
\int (1 - \beta(\tau)) \hat \chi(\tau /T) e^{-i\tau \lambda} \cos(\tau \lap {\tilde g})(\tilde x, \tilde y) = O_T(\lambda^{-N})
\]
for arbitrary $N$, and hence the contribution of the identity term in \eqref{unicoverbound} is trivially bounded by $C_T \lambda^{-1/2}$. We now need only show
\begin{equation} \label{after kernel}
\left| \lambda^{1/2} \sum_{\alpha \in \Gamma \setminus I} \iint b(s,t) w(\tilde \gamma_\alpha(s), \tilde \gamma(t)) a_\pm(T,\lambda; \phi_\alpha(s,t)) e^{\pm i \lambda \phi_\alpha(s,t)} \, ds \, dt \right| \leq C + C_T \lambda^{-1/2}
\end{equation}
where $\phi_\alpha(s,t) = d_{\tilde g}(\tilde \gamma_\alpha(s), \tilde \gamma(t))$.
To do so, we will split the sum into two parts and bound them separately. Fix $R$ to be determined later (in the proof of Proposition \ref{global bound}), and set
\begin{equation} \label{def A}
A = \{ \alpha \in \Gamma : \phi_\alpha(s,t) \leq R \text{ for some } (s,t) \in \supp b \times \supp b \}.
\end{equation}
We will show that the contribution of $A$ to the sum in \eqref{after kernel} is bounded by a constant, and that
\[
 \left| \lambda^{1/2} \sum_{\Gamma \setminus A}  \iint b(s,t) w(\tilde \gamma_\alpha(s),\tilde \gamma(t)) a_\pm(T,\lambda; \phi_\alpha(s,t)) e^{\pm i \lambda \phi_\alpha(s,t)} \, ds \, dt \right| \leq C_T \lambda^{-1/2}.
\]
The above bounds follow from the Propositions \ref{local bound} and \ref{global bound} below, respectively, then follows \eqref{after kernel} and hence Theorem \ref{main theorem}.

\begin{proposition} \label{local bound}
For any fixed $\alpha \in A \setminus I$, there exists a constant $C_\alpha$ such that
\begin{equation}\label{local bound eq}
\left|  \iint b(s,t) w(\tilde \gamma_\alpha(s), \tilde \gamma(t)) a_\pm(T,\lambda; \phi_\alpha(s,t)) e^{\pm i \lambda \phi_\alpha(s,t)} \, ds \, dt \right| \leq C_\alpha \lambda^{-1/2}.
\end{equation}
\end{proposition}

\begin{proposition} \label{global bound} For any fixed $\alpha \in \Gamma \setminus A$, there exists a constant $C_\alpha$ such that
\begin{equation} \label{global bound eq}
\left|  \iint b(s,t) w(\tilde \gamma_\alpha(s), \tilde \gamma(t)) a_\pm(T,\lambda; \phi_\alpha(s,t)) e^{\pm i \lambda \phi_\alpha(s,t)} \, ds \, dt \right| \leq C_\alpha \lambda^{-1}.
\end{equation}
\end{proposition}

\newsection{Phase function bounds} \label{PHASE FUNCTION BOUNDS}

To prove Propositions \ref{local bound} and \ref{global bound}, we will need bounds on the derivatives of the phase function $\phi_\alpha$ for $\alpha \neq I$. First, we bound the mixed partial derivative $\partial_s \partial_t \phi_\alpha$, and second compute $\partial_s^2 \phi_\alpha$ in terms of the curvature $\kappa_\gamma$ of $\gamma$ and the curvature of circles. We will use these computations later to obtain bounds on the pure second derivatives of $\phi_\alpha$.

Let $F : \supp b \times \supp b \times \R$ be the smooth map defined so that $r \mapsto F(s,t,r)$ is the constant-speed geodesic with $F(s,t,0) = \tgamma(t)$ and $F(s,t,1) = \tgamma_\alpha(s)$. If $\partial_s$, $\partial_t$, and $\partial_r$ are the coordinate vector fields living in the domain of $F$, then the Lie brackets $[\partial_s, \partial_t]$, $[\partial_s,\partial_r]$ and $[\partial_t,\partial_r]$ all vanish. Hence,
\begin{equation}\label{st commute}
    \frac{D}{ds} \partial_t F - \frac{D}{dt} \partial_s F = [\partial_s F, \partial_t F] = [F_* \partial_s , F_* \partial_t] = F_*[\partial_s, \partial_t] = 0
\end{equation}
and
\begin{equation} \label{r commute}
    \frac{D}{ds} \partial_r F - \frac{D}{dr} \partial_s F = 0 \quad \text{ and } \quad \frac{D}{dt} \partial_r F - \frac{D}{dr} \partial_t F = 0
\end{equation}
similarly (see for example do Carmo ~\cite{do Carmo}). Now,
\[
    \phi_\alpha^2(s,t) = \int_0^1 |\partial_r F(s,t,r)|^2 \, dr,
\]
and so
\begin{align}
    \nonumber \phi_\alpha(s,t) \partial_s \phi_\alpha(s,t) &= \int_0^1 \left\langle \frac{D}{ds} \partial_r F(s,t,r), \partial_r F(s,t,r) \right\rangle \, dr\\
    \nonumber &= \int_0^1 \partial_r \langle \partial_s F(s,t,r), \partial_r F(s,t,r) \rangle \, dr\\
    \label{first derivative} &= \langle \tgamma_\alpha'(s), \partial_r F(s,t,1) \rangle
\end{align}
where the second line follows from \eqref{r commute} and the geodesic equation $\frac{D}{dr} \partial_r F = 0$, and the third line by the fundamental theorem of calculus. Moreover, since the curves $\tgamma$ and $\tgamma_\alpha$ are disjoint, $\phi_\alpha$ is nonvanishing. From this we have the following fact (also noted in ~\cite{ChenSogge} and ~\cite{SXZ}): $\partial_s \phi_\alpha(s,t)$ vanishes if and only if $\tgamma_\alpha$ is perpendicular to the geodesic adjoining $\tgamma_\alpha(s)$ and $\tgamma(t)$. This works similarly where $\partial_t \phi_\alpha$ vanishes, and hence the gradient $\nabla \phi_\alpha(s,t)$ vanishes if and only if $\tgamma$ and $\tgamma_\alpha$ are both perpendicular to the geodesic adjoining $\tgamma_\alpha(s)$ and $\tgamma(t)$. We will appeal to this fact without reference.

Now we compute the mixed partial derivative $\partial_s \partial_t \phi_\alpha$ at a critical point. From \eqref{first derivative} we obtain
\begin{equation} \label{mixed derivative computation}
    \partial_s \phi_\alpha(s,t) \partial_t \phi_\alpha(s,t) + \phi_\alpha(s,t) \partial_t \partial_s \phi_\alpha(s,t) = \left \langle \tgamma_\alpha'(s), \frac{D}{dr} \partial_t F(s,t,1) \right \rangle.
\end{equation}
From this computation we derive a useful bound.

\begin{lemma} \label{mixed bound}
If $(s,t)$ is a critical point of $\phi_\alpha$,
\[
    |\partial_s \partial_t \phi_\alpha(s,t) | \leq \phi_\alpha^{-1}.
\]
\end{lemma}

\begin{proof}
Since both $\partial_s \phi_\alpha$ and $\partial_t \phi_\alpha$ vanish at $(s,t)$, we are done if we can show that the right and side of \eqref{mixed derivative computation} is bounded by $1$. Since $\partial_t F(s,t,1) = 0$ and $\partial_t F(s,t,0) = \tgamma'(t)$ is perpendicular to $\zeta$, $\partial_t F$ is a perpendicular Jacobi field to $r \mapsto F(s,t,r)$. Hence if $r \mapsto w(r)$ is the vector field along $r \mapsto F(s,t,r)$ obtained by a parallel transport of $\tgamma'(t)$, we write
\[
    \partial_t F(s,t,r) = h(r) w(r)
\]
where $h$ is a smooth function satisfying
\[
    h''(r) + K(F(s,t,r))h(r) = 0
\]
where $K$ is the sectional curvature of $(\R^2,\tilde g)$, with initial conditions
\[
    h(0) = 1 \quad \text{ and } \quad h(1) = 0.
\]
Since $\partial_t F$ must have no conjugate points, $h$ vanishes only at $1$, and so $h$ is nonnegative on $[0,1]$. Since $K \leq 0$, $h'' \geq 0$ on $[0,1]$, and so $h$ is convex. Hence,
\[
    0 \leq h(r) \leq 1 - r \qquad \text{ for } r \in [0,1].
\]
The above line and the limit definition of the derivative yield the bound
\[
    0 \geq h'(1) \geq -1.
\]
Since
\[
    \frac{D}{dr} \partial_t F(s,t,r) = h'(r) w(r),
\]
we have
\[
    \left| \frac{D}{dr} \partial_t F(s,t,1) \right| = 1
\]
which along with the fact $|\tgamma_\alpha| = 1$, yields the desired bound.
\end{proof}

Now we compute $\partial_s^2 \phi_\alpha$. Fix $t_0$ and let $r \mapsto \zeta(s,r)$ denote the unit speed geodesic with $\zeta(s,0) = \tgamma(t_0)$ and $\zeta(s,\phi_\alpha(s,t_0)) = \tgamma_\alpha(s)$. To avoid ambiguity in the notation, we will fix $s_0$ and let $r_0 = \phi_\alpha(s_0,t_0)$, and compute $\partial_s^2 \phi_\alpha(s_0,t_0)$. By \eqref{first derivative}, 
\[
    \partial_s \phi_\alpha(s,t_0) = \left \langle \tgamma_\alpha'(s), \partial_r \zeta(s,\phi_\alpha(s,t_0)) \right \rangle.
\]
Differentiating in $s$ yields
\begin{equation} \label{2t}
    \partial_s^2 \phi_\alpha(s_0,t_0) = \left \langle \frac{D}{ds} \tgamma_\alpha'(s_0), \partial_r \zeta(s_0,r_0) \right \rangle + \left \langle \tgamma_\alpha'(s_0), \left. \frac{D}{ds} \right|_{s = s_0} \partial_r \zeta(s,\phi(s,t_0)) \right \rangle.
\end{equation}
Now
\begin{align*}
    \left. \frac{D}{ds} \right|_{s = s_0} \partial_r \zeta(s,\phi(s,t_0)) &= \frac{D}{ds} \partial_r \zeta(s_0, r_0) + \partial_s \phi_\alpha(s_0,t_0) \frac{D}{dr} \partial_r \zeta (s_0, r_0).
\end{align*}
The latter term on the right vanishes since $r \mapsto \zeta(s_0,r)$ is a geodesic. The curve $s \mapsto \zeta(s,r_0)$ is a geodesic circle of radius $r_0$. Hence, $\partial_r \zeta$ and $\partial_s \zeta$ are perpendicular by Gauss' lemma and $|\partial_r \zeta| = 1$. Hence, there exists a function $\kappa$ such that
\begin{equation} \label{def kappa}
    \frac{D}{ds} \partial_r \zeta = \kappa \partial_s \zeta.
\end{equation}
In fact, $\kappa(s_0,r_0)$ is the geodesic curvature of the circle $s \mapsto \zeta(s,r_0)$ at $s = s_0$.
Hence,
\[
    \frac{D}{ds} \partial_r \zeta(s_0,r_0) = \kappa(s_0,r_0) \partial_s \zeta(s_0, r_0).
\]
and \eqref{2t} becomes
\begin{equation} \label{2t 2}
    \partial_s^2 \phi_\alpha(s_0,t_0) = \left \langle \frac{D}{ds} \tgamma_\alpha'(s_0), \partial_r \zeta(s_0, r_0) \right \rangle + \kappa(s_0,r_0) \left \langle \tgamma_\alpha'(s_0), \partial_s \zeta(s_0,r_0) \right \rangle.
\end{equation}
Let $\theta \in [0, \pi/2]$ denotes the angle of intersection between the curve $\tgamma_\alpha$ and the circle $s \mapsto \zeta(s,r_0)$. We have
\[
    \tgamma_\alpha'(s_0) = \left. \frac{\partial}{\partial s} \right|_{s = s_0} \zeta(s,\phi(s,t_0)) = \partial_s \zeta(s_0,r_0) + \partial_s \phi(s_0,t_0) \partial_r \zeta(s_0,r_0),
\]
and since $\partial_r \zeta$ and $\partial_s \zeta$ are perpendicular,
\[
    \langle \tgamma_\alpha'(s_0) , \partial_s \zeta(s_0,r_0) \rangle = |\partial_s \zeta(s_0,r_0)|^2 = \cos^2(\theta).
\]
The line above and \eqref{2t 2} yields
\begin{equation} \label{pure second}
    \partial_s^2 \phi_\alpha(s_0,t_0) = \cos(\theta) ( \pm \kappa_{\gamma}(s_0) + \cos(\theta) \kappa(s_0,r_0) ),
\end{equation}
where $\pm$ matches the sign of $\langle \frac{D}{ds} \tgamma_\alpha', \partial_r \zeta \rangle$.

\newsection{Curvature of circles}\label{CURVATURE OF CIRCLES}

Fix $t_0$ and let $\zeta$ and $\kappa$ be as in \eqref{def kappa}. To apply \eqref{pure second} in any useful way, we need to know something about the function $\kappa(s,r)$, the curvature of a geodesic circle of radius $r$ centered $\tgamma(t_0)$. Note by the same argument for \eqref{st commute},
\[
    \frac{D}{ds} \partial_r \zeta = \frac{D}{dr} \partial_s \zeta.
\]
This and \eqref{def kappa} yields
\[
    \frac{D^2}{dr^2} \partial_s \zeta = \frac{D}{dr} \left( \kappa \partial_s \zeta \right) = \partial_r \kappa \partial_s \zeta + \kappa \frac{D}{dr} \partial_s \zeta =  (\partial_r \kappa + \kappa^2) \partial_s \zeta.
\]
On the other hand since $\partial_s \zeta$ is a perpendicular Jacobi field along $r \mapsto \zeta(s,r)$,
\[
    \frac{D^2}{dr^2} \partial_s \zeta = -K \partial_s.
\]
Putting these together, we obtain a simple equation for $\kappa$,
\begin{equation} \label{kappa ode}
    \partial_r \kappa + \kappa^2 + K = 0.
\end{equation}

We want to compare the behaviors of $\kappa$ and the quantity $\bfk$, but first we must verify Definition \ref{def k}.

\begin{proposition}\label{verify def k}
    $\bfk(v)$ as given in Definition \ref{def k} exists and is unique for each $v$. Moreover $\bfk$ is a continuous, non-negative function on the unit sphere bundle $SM$.
\end{proposition}

\begin{proof} 
   We use the notation in Definition \ref{def k}. We first observe that $\mathbf k(v)$ does not depend on our choice of $J(0)$. The only other Jacobi field satisfying \eqref{J initial condition} and \eqref{J bounded} is $-J$, which would yield the same value of $\mathbf k(v)$ in \eqref{J' initial condition}. Granted, this holds if for a fixed choice of $J(0)$, there is exactly one value for $\frac{D}{dr} J(0)$ such that $J$ satisfies \eqref{J bounded}. We prove this now.

    Since both $J(0)$ and $\frac{D}{dr} J(0)$ are perpendicular to $\zeta'(0)$,
    \[
        \langle J(r), \zeta'(r) \rangle = 0 \quad \text{ for all } r.
    \]
    Hence, if $w(r)$ denotes the vector at $\zeta(r)$ obtained through a parallel transport of $J(0)$ along $\zeta$, we write
    \[
        J(r) = h(r) w(r)
    \]
    for some smooth function $h$ satisfying
    \begin{align}
        \label{h jacobi equation} h'' + Kh &= 0 \quad \text{ and }\\
        \label{h initial condition 0} h(0) &= 1,
    \end{align}
    where $K = K(\zeta(r))$ is the sectional curvature at $\zeta(r)$. We have reduced the problem to proving that there exists a unique function $h$ satisfying \eqref{h jacobi equation} and \eqref{h initial condition 0} and also
    \begin{equation} \label{h bounded}
        |h(r)| \leq C \quad \text{ for } r \leq 0,
    \end{equation}
    and then setting $\mathbf k(v) = h'(0)$. We begin with uniqueness. Suppose $h_1$ and $h_2$ both satisfy \eqref{h jacobi equation}, \eqref{h initial condition 0}, and \eqref{h bounded}. Then the difference $u = h_1 - h_2$ satisfies \eqref{h jacobi equation} and \eqref{h bounded} with initial data $u(0) = 0$. If $u'(0) > 0$, then $u(r) > 0$ for all $r > 0$, otherwise we would have a conjugate point. Hence, $u'' \geq 0$, and so $u(r) \geq ru'(0)$, a contradiction. We derive a similar contradiction if $u'(0) < 0$ and conclude $u'(0) = 0$.

    To prove existence, we construct a bounded $h$ as a limit. For all $s > 0$, let $h_{s}$ denote the unique function satisfying \eqref{h jacobi equation}, \eqref{h initial condition 0}, and $h_s(-s) = 0$. We construct as a limit
    \begin{equation} \label{h limit}
        h_{-\infty} = \lim_{s \to \infty} h_s = h_{1} + \int_1^\infty \partial_s h_s \, ds
    \end{equation}
    which we will show converges uniformly on compact sets. We then have smooth convergence by \eqref{h jacobi equation}. Hence, $h_{\infty}$ satisfies \eqref{h jacobi equation} and \eqref{h initial condition 0}. It will then be left to show that $h_{\infty}$ satisfies \eqref{h bounded}. To prove convergence, we first show
    \begin{equation} \label{h t derivative bound}
        |\partial_s h_s(r)| \leq -\frac{r}{s^2} \quad \text{ for } r \leq 0.
    \end{equation}
    Now $h_s$ may only vanish at $-s$, otherwise we have conjugate points. Hence, $h_s > 0$ on $(-s,0]$, and so $h_s'' \geq 0$ on $[-s,0]$. Since $h_s(0) = 1$ and $h_s(-s) = 0$,
    \[
        0 \leq h_s(r) \leq \left( 1 + \frac{r}{s} \right) \quad \text{ for } -s \leq r \leq 0
    \]
    by convexity. We conclude that
    \[
        0 \leq h_s'(-s) \leq \frac{1}{s}
    \]
    by writing $h_s'(-s)$ as a difference quotient and applying the previous inequality. Now since $h_s(-s) = 0$ for all $s > 0$,
    \[
        0 = \frac{d}{ds} h_s(-s) = -h_s'(-s) + \partial_s h_s(-s).
    \]
    Hence,
    \begin{equation} \label{dt ht endpoint bound}
        0 \leq \partial_s h_s(-s) \leq \frac{1}{s}.
    \end{equation}
    Now $\partial_s h_s$ also satisfies \eqref{h jacobi equation} with initial data $\partial_s h_s(0) = 0$. Since $\partial_s h_s(-s) > 0$, a similar convexity argument as before yields bounds
    \[
        0 < \partial_s h_s(r) \leq -\partial_s h_s(-s) \frac{r}{s} \quad \text{ for } -s \leq r < 0.
    \]
    \eqref{h t derivative bound} follows from the above inequality and \eqref{dt ht endpoint bound}. The bound \eqref{h t derivative bound} implies the pointwise convergence of the limit \eqref{h limit}. Moreover if we fix $s_0 > 0$, for $r \in [-s_0,0]$ we have
    \begin{equation}\label{h error}
        |h_\infty(r) - h_{s_0}(r)| = \left| \int_{s_0}^\infty \partial_s h_s(r) \, ds \right| \leq -\int_{s_0}^\infty \frac{r}{s^2} \, ds = -\frac{r}{s_0}.
    \end{equation}
    This implies uniform convergence on compact sets. Similarly,
    \[
        h_\infty(-s_0) = \int_{s_0}^\infty \partial_s h_s(-s_0) \, ds,
    \]
    which together with \eqref{h t derivative bound} implies
    \begin{equation} \label{true h bounds}
        0 < h_\infty(-s_0) \leq 1 \quad \text{ for } s_0 > 0.
    \end{equation}
    which is stronger than \eqref{h bounded}.
    This completes the proof of existence.

    To show that $\mathbf k(v)$ is non-negative, we argue that $h_\infty'(0) \leq 0$. By \eqref{true h bounds}, $h_\infty(r)$ does not vanish for $r > 0$, and so $h_\infty''(r) \geq 0$ for $r \geq 0$. However if at the same time $h_\infty'(0) > 0$, $h_\infty$ would certainly be unbounded on $[0,\infty)$. Hence $h_\infty'(0) \leq 0$ as desired.

    Finally, we show $\mathbf k$ is continuous on $SM$. To do so, we show that $\mathbf k$ is continuous on every continuous path $t \mapsto v(t)$ in $SM$. If $r \mapsto \zeta(t,r)$ is the geodesic with $\partial_r \zeta(t,0) = v(t)$, we let $h_\infty(t,r)$ and $h_s(t,r)$ be as constructed above along the geodesic $r \mapsto \zeta(t,r)$. Now in the limit as $t \to 0$, the sectional curvature $K(\zeta(t,r))$ converges to $K(\zeta(0,r))$ uniformly for $r$ in a compact set. Combined with \eqref{h jacobi equation}, we have for any $\epsilon > 0$ and $s > 0$ a $\delta > 0$ such that
    \[
        |h_s(t,r) - h_s(0,r)| < \frac{\epsilon}{3}  \quad \text{ for } -s \leq r \leq 0
    \]
    if $|t| < \delta$. Moreover if $r$ lies in some compact set, by \eqref{h error} there exists $s > 0$ large enough such that
    \[
        |h_\infty(t,r) - h_s(t,r) | < \frac{\epsilon}{3}
    \]
    independently of $t$. Putting these bounds together, we have
    \begin{align*}
        |h_\infty(t,r) - &h_\infty(0,r)|\\
        &\leq |h_\infty(t,r) - h_s(t,r)|  + |h_s(t,r) - h_s(0,r)|  + |h_s(0,r) - h_\infty(0,r)| < \epsilon,
    \end{align*}
    i.e. $h_\infty(t,r) \to h_\infty(0,r)$ uniformly for $r$ in a compact set. By \eqref{h jacobi equation}, $\partial_r^2 h_\infty(t,r) \to \partial_r^2 h_\infty(0,r)$ uniformly for $r$ in a compact set. Hence, $\partial_r h(t,r) \to \partial_r h(0,r)$ as $t \to 0$, and in particular $\mathbf k(v(t)) \to \mathbf k(v(0))$.
\end{proof}

\begin{remark} \label{k lift}
    Let $\tilde \bfk$ be given by $\tilde \bfk(\tilde v) = \bfk(v)$ where $\tilde v$ is a lift of $v$ to $S\R^2$. Since the covering map is a local isomorphism, $\tilde \bfk$ satisfies Definition \ref{def k} on the manifold $(\R^2,\tilde g)$. From now on we will work exclusively in the universal cover, noting that $\bfk$ in the hypotheses \eqref{main theorem hypotheses} can be freely replaced with $\tilde \bfk$.
\end{remark}

We can loosen Definition \ref{def k} a little bit. If $v$, $\zeta$, and $J$ are as in Definition \ref{def k} (here we replace the manifold $M$ in the definition with the universal cover as justified by the above remark), except that $|J(0)|$ is allowed to take any value except $0$, we may write
\[
    \tilde \bfk(v) = \frac{|\frac{D}{dr} J(0)|}{|J(0)|}.
\]
Then,
\[
    \tilde \bfk(\zeta'(r)) = \frac{|\frac{D}{dr} J(r)|}{|J(r)|} = \frac{h'(r)}{h(r)} \qquad \text{ for all } r \geq 0
\]
using $h$ as in the proof of Proposition \ref{verify def k}. It follows that
\begin{equation} \label{k ode}
    \frac{d}{dr} \tilde \bfk(\zeta'(r)) = \frac{h''(r)}{h(r)} - \frac{h'(r)^2}{h(r)^2} = -K - \tilde \bfk(\zeta'(r))^2,
\end{equation}
and hence $\tilde \bfk(\zeta'(r))$ satisfies the same ordinary differential equation \eqref{kappa ode} as $\kappa$. As a consequence, we have the following lemma.

\begin{lemma} \label{large radius}
    Let $r \mapsto \zeta(r)$ be a unit-speed geodesic in $(\R^2,\tilde g)$ and $\kappa(r)$ the geodesic curvature at $\zeta(r)$ of the circle of radius $r$ with center at $\zeta(0)$. Then,
    \[
        0 < \kappa(r) - \tilde \bfk(\partial_r \zeta(r)) \leq r^{-1}, \qquad r > 1.
    \]
\end{lemma}

\begin{proof}
    Since both $\kappa$ and $\tilde \bfk$ satisfy \eqref{kappa ode}, the difference $\kappa - \tilde \bfk$ satisfies
    \begin{equation} \label{difference ode}
        \partial_r (\kappa - \tilde \bfk) = -(\kappa^2 - \tilde \bfk^2).
    \end{equation}
    Since $\kappa(r)$ is large for small $r$, we can easily guarantee that $\kappa(r_0) > \tilde \bfk(\zeta'(r))$ for $0 < r \ll 1$. Now $\kappa$ and $\tilde \bfk$ are smooth for $r > 0$, and since $\kappa - \tilde \bfk = 0$ is an equilibrium of \eqref{difference ode}, we have that
    \[
        \kappa(r) - \tilde \bfk(\zeta'(r)) > 0 \qquad \text{ for all } r > 0.
    \]
    Hence
    \[
        \partial_r (\kappa - \tilde \bfk) = -\frac{\kappa + \tilde \bfk}{\kappa - \tilde \bfk}(\kappa - \tilde \bfk)^2 \leq -(\kappa - \tilde \bfk)^2,
    \]
    the inequality a consequence of the fact that $\kappa > \tilde \bfk$. By comparison,
    \[
        \kappa(r) - \tilde \bfk(\zeta'(r)) \leq r^{-1},
    \]
    as desired.
\end{proof}

\newsection{Conclusion of the proof of Theorem \ref{main theorem}}

All that is left is to prove Propositions \ref{local bound} and \ref{global bound}

\begin{proof}[Proof of Proposition \ref{local bound}]
    Since $A$ is fixed and finite, we may restrict the support of $b$ without worrying about doing so uniformly over elements of $A$. Fix $\alpha \in A \setminus I$. Let $D$ denote the diagonal of $\supp b \times \supp b$. We claim that that
    \begin{equation} \label{pf medium time 1}
        D \subset \{ \partial_t^2 \phi_\alpha \neq 0\} \cup \{ \partial_s^2 \phi_\alpha \neq 0\} \cup \{ \nabla \phi_\alpha \neq 0 \}
    \end{equation}
    Provided our claim is true, we restrict the support of $b$ by a fine enough partition of unity so that at least one of the conditions $\partial_t^2 \phi_\alpha \neq 0$, $\partial_s^2 \phi_\alpha \neq 0$, or $\nabla \phi_\alpha \neq 0$ holds on all of $\supp b \times \supp b$. In the first case, the proposition follows by stationary phase ~\cite[Theorem 1.1.1]{FIOs} in $t$, and similarly for the second case. In the third case, the proposition follows by nonstationary phase ~\cite[Lemma 0.4.7]{FIOs}.
    
    Fix $s_0 = t_0 \in \supp b$ and suppose $\nabla \phi_\alpha(s_0,t_0) = 0$. To prove \eqref{pf medium time 1}, we need only show that either $\partial_s^2 \phi_\alpha(s_0,t_0) \neq 0$ or $\partial_t^2 \phi_\alpha(s_0,t_0) \neq 0$. Let $r \mapsto \zeta(r)$ be the constant-speed geodesic with $\zeta(0) = \tgamma(t_0)$ and $\zeta(1) = \tgamma_\alpha(s_0)$. By the computation \eqref{pure second},
    \[
        \partial_s^2 \phi_\alpha(s_0,t_0) = \pm \kappa_\gamma(s_0) + \kappa(s_0,\phi_\alpha(s_0,t_0))
    \]
    where $\pm$ agrees with the sign of $\langle \zeta'(1), D/ds \tgamma_\alpha(s_0) \rangle$. If $\langle \zeta'(1), D/ds \tgamma_\alpha'(s_0) \rangle \geq 0$, we are done since $\kappa(s_0,\phi_\alpha(s_0,t_0))$ is positive. If not, we will prove that
    \[
        \langle -\zeta'(0), D/dt \tgamma'(t_0) \rangle \geq 0,
    \]
    which yields $\partial_t^2 \phi_\alpha(s_0,t_0) > 0$ from a similar computation as \eqref{pure second} in $t$. Since $\alpha$ is an isometry,
    \begin{align*}
        \left\langle -\zeta'(0), \frac{D}{dt} \tgamma'(t_0) \right\rangle = -\left\langle \alpha_* \zeta'(0), \alpha_* \frac{D}{dt} \tgamma'(t_0) \right\rangle = -\left\langle \alpha_* \zeta'(0), \frac{D}{ds} \tgamma_\alpha'(s_0) \right\rangle.
    \end{align*}
    We claim that
    \[
        \zeta'(1) = \alpha_* \zeta'(0).
    \]
    As noted earlier, $\zeta$ is perpendicular to both $\gamma$ and $\gamma_\alpha$ since $\nabla \phi_\alpha(s_0,t_0) = 0$. Since $\alpha_*(\tgamma'(t_0)) = \tgamma_\alpha'(s_0)$, we need only rule out the possibility that $-\zeta'(1) = \alpha_* \zeta'(0)$. If this were the case, however, $\zeta(1/2) = \alpha \circ \zeta(1/2)$ by uniqueness, contradicting the fact that $\alpha$ is a deck transformation. Hence,
    \[
        -\left\langle \alpha_* \zeta'(0), \frac{D}{ds} \tgamma_\alpha'(s_0) \right\rangle = -\left\langle \zeta'(1), \frac{D}{ds} \tgamma_\alpha'(s_0) \right\rangle > 0,
    \]
    as desired.
\end{proof}

\begin{proof}[Proof of Proposition \ref{global bound}]
    By our hypothesis \eqref{main theorem hypotheses} on the curvature of $\gamma$, and since $\bfk$ is continuous, we restrict the support of $b$ so that
    \[
        \inf_{s,t \in \supp b} |\kappa_\gamma(t) - \bfk(\pm \gamma'^\perp(s))| > 2\epsilon
    \]
    for some small $\epsilon > 0$. Let $\zeta$ be defined as in Section \ref{PHASE FUNCTION BOUNDS}, that is let $r \mapsto \zeta(s,t,r)$ be the unit-speed geodesic with $\zeta(s,t,0) = \tgamma(t)$ and $\zeta(s,t,\phi_\alpha(s,t)) = \tgamma_\alpha(s)$. Moreover let $\kappa(s,t)$ denote the curvature at $\tgamma_\alpha(s)$ of the circle with center $\tgamma(t)$ and radius $\phi_\alpha(s,t)$. We set $R = 2\epsilon^{-1}$ in \eqref{def A}. Since $R > \epsilon$, Lemma \ref{large radius} tells us
    \[
        |\kappa(s,t) - \tilde \bfk(\zeta'(s,t,\phi_\alpha(s,t)))| < \epsilon,
    \]
    and hence
    \begin{equation} \label{kappa difference}
        |\kappa(s,t) - \kappa_{\gamma_\alpha}(t)| > \epsilon \qquad \text{ for } s,t \in \mathcal I.
    \end{equation}
    We claim that the determinant of the Hessian of $\phi_\alpha$ is nonzero at critical points of $\phi$. It follows that, for each $\alpha \in \Gamma \setminus A$, $\supp b \times \supp b$ contains finitely many stationary points of $\phi_\alpha$, all of which are non-degenerate. The desired bound follows by stationary phase ~\cite[Theorem 1.1.4]{FIOs}.
    
    Suppose $\nabla \phi_\alpha(s,t) = 0$ at some point $(s,t) \in \supp b \times \supp b$. By the bound \eqref{mixed bound} and our assertion that $R = 2\epsilon^{-1}$, we have
    \[
        |\partial_t \partial_s \phi_\alpha(s,t)| \leq \epsilon/2.
    \]
    Moreover, by \eqref{kappa difference} and the computation \eqref{pure second}, we have
    \[
        |\partial_s^2 \phi_\alpha(s,t)| \geq \epsilon \quad \text{ and } \quad |\partial_t^2 \phi_\alpha(s,t)| \geq \epsilon.
    \]
    Hence, the determinant of the hessian $|\det \phi''(s,t)|$ is bounded by
    \begin{align*}
        |\det \phi''(s,t)| &= |(\partial_s^2 \phi_\alpha(s,t))(\partial_t^2 \phi_\alpha(s,t)) - (\partial_s \partial_t \phi_\alpha(s,t))^2|\\
        &\geq |\partial_s^2 \phi_\alpha(s,t)||\partial_t^2 \phi_\alpha(s,t)| - |\partial_s \partial_t \phi_\alpha(s,t)|^2\\
        &\geq \frac{3}{4}\epsilon^2,
    \end{align*}
    which proves our claim.
\end{proof}


\end{document}